\documentclass[10pt,a4paper]{amsart}
\usepackage{amsmath,amssymb, amsbsy}
\usepackage{color,psfrag}
\usepackage[dvips]{graphicx}
\usepackage[colorlinks=true]{hyperref}
\usepackage{mathtools}
\usepackage{amsmath}
\usepackage{amsthm}
\usepackage{amssymb}

\newcommand{\nero}{\color{black}}

\newcommand{\R}{\mathbb{R}}

\newcommand{\hn}{\mathbb{H}^{N}}
\newcommand{\tq}{\mathbb{T}_{q+1}}

\numberwithin{equation}{section}

\DeclareFontFamily{U}{mathx}{\hyphenchar\font45}
\DeclareFontShape{U}{mathx}{m}{n}{
      <5> <6> <7> <8> <9> <10>
      <10.95> <12> <14.4> <17.28> <20.74> <24.88>
      mathx10
      }{}
\DeclareSymbolFont{mathx}{U}{mathx}{m}{n}
\DeclareFontSubstitution{U}{mathx}{m}{n}
\DeclareMathAccent{\widecheck}{0}{mathx}{"71}
\DeclareMathAccent{\wideparen}{0}{mathx}{"75}

\makeatletter
\newcommand{\leqnomode}{\tagsleft@true}
\newcommand{\reqnomode}{\tagsleft@false}
\makeatother

\newtheorem{theorem}{Theorem}[section]
\newtheorem{proposition}[theorem]{Proposition}

\newtheorem{corollary}[theorem]{Corollary}

\newtheorem{remark}[theorem]{Remark}
\newtheorem{definition}[theorem]{Definition}

\title[Poincar\'{e} and Hardy inequalities on homogeneous trees]{ Poincar\'{e} and Hardy inequalities on homogeneous trees} 
\author[Berchio]{Elvise Berchio}
\address{Dipartimento di Scienze Matematiche ``Giuseppe Luigi Lagrange'',
  Politecnico di Torino, Corso Duca degli Abruzzi 24, 10129 Torino,
  Italy - Dipartimento di Eccellenza 2018-2022}
\email{elvise.berchio@polito.it}

\author[Santagati]{Federico Santagati}
\address{Dipartimento di Scienze Matematiche ``Giuseppe Luigi Lagrange'',
  Politecnico di Torino, Corso Duca degli Abruzzi 24, 10129 Torino,
  Italy - Dipartimento di Eccellenza 2018-2022}
\email{federico.santagati@polito.it}

\author[Vallarino]{Maria Vallarino}
\address{Dipartimento di Scienze Matematiche ``Giuseppe Luigi Lagrange'',
  Politecnico di Torino, Corso Duca degli Abruzzi 24, 10129 Torino,
  Italy - Dipartimento di Eccellenza 2018-2022}
\email{maria.vallarino@polito.it}



\begin{document}
\maketitle
\begin{abstract}
We study Hardy-type inequalities on infinite homogeneous trees. More precisely, we derive optimal Hardy weights for the combinatorial Laplacian in this setting and we obtain, as a consequence, optimal improvements for the Poincar\'e inequality. 
\end{abstract}

%

\section{Introduction}

 Given a linear, elliptic, second-order, symmetric nonnegative operator $P$ on $\Omega$, where $\Omega$ is a (e.g. Euclidean) domain, a \emph{Hardy weight} is  a nonnegative function $W$ such that the following inequality holds
 \begin{equation}\label{1}
 q(u)\ge \int_\Omega Wu^2\,{\rm d}x\ \ \ \forall u\in C_c^\infty(\Omega),
 \end{equation}
where $q(u)=\langle u,Pu\rangle$ is the quadratic form associated to $P$. Clearly, the final (and most ambitious) goal is to get weights $W$ such that inequality \eqref{1} is not valid for $V>W$, i.e. the operator $P-W$ is \emph{critical} in the sense of \cite[Definition 2.1]{pinch}.  When $P=-\Delta$ is the Laplace--Beltrami operator on a Riemannian manifold, the problem of the existence of Hardy weights has been widely studied in the literature, either in the Euclidean setting, see e.g \cite{GFT,BrezisM, Brezis, gaz, Kufner, MMP, Mitidieri} or on general manifolds, see e.g. \cite{Carron,Dambrosio, pinch, Kombe2, AKR,ngo, Yang}. Recently, the attention has also been devoted to the discrete setting, see e.g. \cite{Golenia, pinchoverA, pinchoverC, pinchover} and references therein. 

\par

The present paper is motivated by some recent results obtained in \cite{BGG}, see also \cite{AK}  and \cite{BGGP}, within the context of Cartan--Hadamard manifolds $M$. In particular, when $M$ is the hyperbolic space $\hn$, i.e. the simplest example of manifold with negative sectional curvature, the following Hardy weight has been determined for $P$ being the Laplace--Beltrami operator $-\Delta_{\mathbb H^N}$ on $\hn$ with $ N \geq 3$:
$$W(r)=\frac{(N-1)^2}4-\frac1{4\,r^2}-\frac{(N-1)(N-3)}{4}  \frac{1}{\sinh^2 r},$$
where $r=d(o,x)>0$ denotes the geodesic distance of $x$ from a fixed pole $o\in \hn$. Besides, it is proved that the operator $-\Delta_{\hn}-W$ is {\it critical} in $\mathbb{H}^{N}\setminus\{o\}$. It is worth noticing that the number $\frac{(N-1)^2}4$ in $W(r)$ coincides with the bottom of the $L^2$-spectrum of $-\Delta_{\hn}$. Hence, the existence of the above weight yields the following improved Poincar\'e inequality:
 \begin{equation*}\label{poincareeq}
 \begin{aligned}
\int_{\hn} |\nabla_{\hn} u|^2 \ {\rm d}v_{\hn} - \frac{(N-1)^2}{4} \int_{\hn} u^2 \ {\rm d}v_{\hn}
\geq  \int_{\hn} \, R\, u^2 \ {\rm d}v_{\hn} \ \ \ \forall u\in C_c^\infty(\hn),
\end{aligned}\end{equation*}
where the remainder term is 
\begin{equation}\label{resto_iperb}
R(r)=\frac1{4\,r^2}+\frac{(N-1)(N-3)}{4}  \frac{1}{\sinh^2 r} \sim \frac1{4\,r^2}  \quad \text{as } r\rightarrow + \infty,
\end{equation}
and, as a consequence of the criticality issue, all constants in \eqref{resto_iperb} turn out to be sharp.\par

Let  $\Gamma=(V,E)$ denote a locally finite graph, where $V$ and $E$ denote a countably infinite set of vertices and the set of edges respectively. We recall that the combinatorial Laplacian $\Delta$ of a function $f$ in the set $C(V)$ of real valued functions defined on $V$ is defined by 
$$
    \Delta f(x) := \sum_{y \sim x} \bigg{(}f(x) - f(y)\bigg{)}= m(x)f(x) - \sum_{y \sim x} f(y)\qquad  \forall x \in V  \,,
$$
where $m(x)$ is the degree of $x$, i.e. the number of neighbors of $x$. The existence of Hardy weights for the combinatorial Laplacian or for more general operators on graphs has been recently studied in literature (see again \cite{Golenia, pinchoverA, pinchoverC, pinchover}).

\smallskip 
We set our analysis on the case where the graph $\Gamma$ is the homogeneous tree $\mathbb{T}_{q+1}$, i.e. a connected graph with no loops such that every vertex has $q+1$ neighbours, and we focus on the transient case, namely we always assume $q \ge 2$. $\tq$ has been the object of investigation of many papers either in the field of harmonic analysis or of PDEs, see e.g. to \cite{AMPS, A, BW, CMW, CCPS, CCS, meda, CMS, FTN, PP}. In particular, the homogeneous tree is in many respects a discrete analogue of the hyperbolic plane; we refer the reader to \cite{BW} for a discussion on this point. Therefore, since $\tq$ is the basic example of graph of exponential growth, as $\hn$ is the basic example of Riemannian manifold with exponential growth, it is natural to investigate whether the above mentioned results in $\hn$ have a counterpart in $\tq$: this will be the main goal of the paper. \par
 
In $\mathbb{T}_{q+1}$ the operator $\Delta$ is bounded on $\ell^2$ and  its $\ell^2$-spectrum is given by $ [({q}^{1/2}-1)^2, ({q}^{1/2}+1)^2]$ (see \cite{meda}). Hence the following Poincar\'e inequality holds 
\begin{equation*}
   \frac{1}{2} \sum\limits_{\substack{x,y \in \mathbb{T}_{q+1}\\
    x \sim y}} \bigg(\varphi(x)-\varphi(y) \bigg)^2  \ge \Lambda_q \sum_{x \in \mathbb{T}_{q+1}} \varphi^2(x) \qquad \forall \varphi \in C_0(\mathbb{T}_{q+1}),
\end{equation*}
with $\Lambda_q:=(q^{1/2}-1)^2$. \\ 
By \cite[Theorem 0.2]{pinchover} a Hardy weight for $\Delta$ on a transient graph $\Gamma$,  is given by 
$ W_{opt}=\frac{\Delta G_o^{1/2}}{G_o^{1/2}}$, where $G_o(x):=G(x,o)$ is the positive minimal Green function and $o$ is a fixed point.  Furthermore, $W_{opt}$ is optimal in the sense of Definition \ref{defopt} below and this implies, in particular, that the operator $\Delta-W_{opt}$ is { \it critical}. If $\Gamma=\tq$, then the function $G_o$ can be written explicitly, see Proposition \ref{optimal1} below, and $W_{opt}$ reads as follows: 
 \begin{align}\label{numerostellaintro}
  W_{opt}(x)=\begin{cases} \Lambda_q+q^{1/2}-q^{-1/2} &\text{if $|x|=0$,} \\ 
 \Lambda_q &\text{if $|x| \ge 1$.} \end{cases}
  \end{align}
By exploiting the super-solutions technique, in the present paper we provide the following new family of Hardy weights for $\Delta$ on $\tq$:
\begin{align*}
    W_{ \beta, \gamma}(x) = \begin{cases} 
    q+1-q^{1/2}(\frac{1}{\gamma}+\frac{1}{\gamma \,q})  \nero &\text{if $|x| = 0$, } \\ 
q+1-{q^{1/2}(2^\beta+\gamma}) &\text{if $|x|=1$,} \\
q+1-{q^{1/2}[(1+\frac{1}{|x|})^\beta}+(1-\frac{1}{|x|})^\beta] &\text{if $|x|\ge 2$,} \end{cases}
\end{align*}
where $0 \le \beta \le \log_2 q^{1/2}$ and ${q}^{-1/2}\leq \gamma \leq {q}^{-1/2} + {q}^{1/2}-2^\beta$. Moreover, if $\beta=1/2$ we prove that the weight $W_{1/2, \gamma}$ is optimal (see again Definition \ref{defopt}), hence the operator $\Delta-W_{1/2, \gamma}$ is {\it critical}. We notice that
\begin{align*}
   W_{\beta,\gamma}(x)=\Lambda_q+q^{1/2}\frac{\beta (1-\beta)}{|x|^2}+o\Big{(}\frac{1}{|x|^2}\Big{)} \qquad \text{as $|x| \to \infty$\,,} 
\end{align*}
hence the slowest decay at infinity occurs exactly for $\beta=1/2$. 
\par
It is readily seen that the quadratic form inequality associated to $\Delta-W_{opt}$ in \eqref{numerostellaintro}  can be read as an (optimal) {\it local} improvement of the Poincar\'e inequality on $\tq$ at $o$.  A direct inspection reveals that the weights $W_{ \beta, \gamma}$ satisfy $W_{ \beta, \gamma}>\Lambda_q$ on $ \mathbb{T}_{q+1}$ for all $0 \le \beta \le \log_2\big{(}\frac{3}{2}-\frac{1}{2q}\big{)}$ and $\frac12 + \frac{1}{2q} \le \gamma \le 2-2^\beta $. Hence, for such values of $\beta$ and $\gamma$, we derive the following family of {\it global} improved Poincar\'e inequalities:
\begin{align}\label{asterisco2intro}
   \frac 12 \sum\limits_{\substack{x,y \in \mathbb{T}_{q+1} \\ x \sim y}}\bigg{(}\varphi(x)-\varphi(y)\bigg{)}^2\ge \sum_{x \in \mathbb{T}_{q+1}} \Lambda_q\varphi^2(x) + \sum_{x \in \mathbb{T}_{q+1}} R_{\beta, \gamma}(x) \varphi^2(x) \qquad \forall \varphi \in C_0(\mathbb{T}_{q+1}),
\end{align}
where 
\begin{align*}
  0 \le R_{\beta, \gamma}(x) = \begin{cases} 
   q^{1/2}(2-\frac{1}{\gamma}-\frac{1}{\gamma \,q} )  \nero &\text{if $|x|=0$,}\\
  q^{1/2}( 2-2^{\beta}-\gamma) \nero &\text{if $|x|=1$,} \\ 
q^{1/2}\bigg{(}2-(1+\frac{1}{|x|})^\beta-(1-\frac{1}{|x|})^\beta\bigg{)} &\text{if $|x| \ge 2$}.
\end{cases}
\end{align*}
It is worth noticing that the maximum of $R_{\beta, \gamma}$ at $o$ is reached by choosing $\gamma$ as large as possible, namely by taking  $\gamma=2-2^\beta$. Since such value is maximum for $\beta=0$, we conclude that, among the weights $W_{ \beta, \gamma}$ improving the Poincar\'e inequality, the largest at $o$ is $W_{ 0, 1}\equiv W_{opt}$.  

Even if \eqref{asterisco2intro} improves globally the Poincar\'e inequality, we do not know whether this improvement is sharp on the whole $\tq$. Nevertheless, a {\it sharp} improvement is provided by the critical weight $ W_{ 1/2, \gamma}$ outside the ball $B_2(o)$. More precisely, there holds
\begin{align*}\label{optimalPoincar}
   \frac 12 \sum\limits_{\substack{x,y \in \mathbb{T}_{q+1} \\ x \sim y}}\bigg{(}\varphi(x)-\varphi(y)\bigg{)}^2\ge \sum_{x \in \mathbb{T}_{q+1}}\Lambda_q \varphi^2(x)+\sum_{x \in \mathbb{T}_{q+1}}\overline{R}(x)\varphi^2(x), \ \ \ \ \ \forall \varphi \in C_0(\mathbb{T}_{q+1}\setminus B_2(o)),
\end{align*}
 where 
\begin{align*}
   \overline{R}(x) =
    q^{1/2}\bigg{[}2-\bigg{(}1+\frac{1}{|x|}\bigg{)}^{1/2}-\bigg{(}1-\frac{1}{|x|}\bigg{)}^{1/2}\bigg{]} \qquad \text{if $|x| \ge 2$}
\end{align*}
and the constant $q^{1/2}$ is sharp. Notice that
$$ \overline{R}(x)\sim     q^{1/2}\, \frac{1}{4|x|^2}  \quad \text{as } |x|\rightarrow + \infty,$$
namely the decay of the remainder term is of the same order of that provided by \eqref{resto_iperb} in $\hn$, thereby confirming the analogy between $\tq$ and $\hn$. 

\par
Following the arguments used in the particular case of a homogeneous tree, in the last part of the paper we find a class of Hardy weights for the combinatorial Laplacian on rapidly growing {\emph{radial}} trees, i.e. trees where the number of neighbours of a vertex $x$ only depends on the distance of $x$ from a fixed vertex $o$. This is a first result which might shed light on future related investigations on more general graphs. 

\par

\bigskip

The paper is organized as follows. In Section 2 we introduce the notation and we state our main results, namely Theorem \ref{optimalHf}, where we provide a family of optimal weights for $\Delta$ on $\mathbb{T}_{q+1},$ and Theorem \ref{2.11} where we state the related improved Poincar\'e inequality. Section 3 is devoted to the proof of the statements of Section 2. Finally, in Section 4 we present a generalization of our results in the context of radial trees.

\section{Notation and main results}
We consider a graph $\Gamma=(V,E)$, where $V$ and $E$ denote a countably infinite set of vertices and the set of edges respectively, with the usual discrete metric $d$. If $(x, y) \in E$ we say that $x$ and $y$ are neighbors and we write $x \sim y$. We assume that $\Gamma$ is a connected graph, that is, for every $x,y \in V$ there exists a finite sequence of vertices $x_1,\dots,x_n$ such that $x_0=x$, $x_n=y$ and $x_j\sim x_{j+1}$ for $j=0,\dots,n-1$. We also require that $(x,y)\in E$ if and only if $ (y,x) \in E$. We use the notation $m(x)$ to indicate the
degree of $x$, that is the number of edges that are attached to $x$ and we assume that $\Gamma$ is locally finite, i.e. $m(x)<\infty$ for all $x \in V$. When a vertex $o \in V$ is fixed let $x\mapsto |x|$ be the function which associates to each vertex $x$ the distance $d(x,o)$ and define $B_r(o)=\{x \ \text{s.t.} \  |x|<r\}$. We denote by $C(V)$ the set of real valued function defined on $V$ and by $C_0(V)$  the subspace consisting on finitely supported functions. Finally, we introduce the space of square summable functions
\begin{equation*}
   \ell^2(V) = \{ f     \in  C(V) \ \text{s.t.} \ \sum_{x \in V}  f^2(x) < + \infty \} .
\end{equation*}
This is a Hilbert space with the inner product 
\begin{equation*}
    \langle f, g \rangle = \sum_{x \in V} f(x)g(x), 
\end{equation*} and the induced norm $\Vert f \Vert = \sqrt{\langle f, f \rangle}$. As shown in \cite{Woich2, Woich}
\begin{equation*}
   \langle \Delta \varphi, \varphi \rangle_{\ell^2} = \frac{1}{2} \sum\limits_{\substack{x,y  \in V \\ x \sim y}} \bigg{(} \varphi(x)-\varphi(y)\bigg{)}^2\qquad \forall \varphi\in C_0(V).
\end{equation*}
More generally, we consider Schr\"odinger operators $H=\Delta +Q$ where $Q$ is any potential. A function $f$ is called $H$-(super)harmonic in $V$ if 
\begin{equation*}
    Hf(x)=0 \ \ \ (Hf(x)\ge0)\qquad \forall x \in V.
\end{equation*}
By Hardy-type inequality for a positive Schr\"odinger operator $H$ we mean an inequality of the form 
\begin{align*}
    \langle H\varphi, \varphi\rangle \ge \langle  W\varphi, \varphi \rangle  \qquad \forall \varphi \in C_0(V),
\end{align*}
where $W \not \equiv 0$ is a nonnegative function in $C(V)$. We write $h(\varphi)$ and $W(\varphi)$ in place of $\langle H \varphi, \varphi \rangle $ and $\langle W\varphi, \varphi \rangle$, respectively. In particular we denote $h_\Delta(\varphi)=\langle \Delta \varphi, \varphi \rangle.$ 
\\ 
In \cite{pinchover} the authors introduce the notion of optimal weight for a Hardy-type inequality; we recall some fundamental definitions that we need in the later discussion.
\begin{definition}
Let $h$ be a quadratic form associated with a Schr\"odinger operator $H$, such that $h \ge 0$ on $C_0(V)$. The form $h$ is called {\bf{subcritical}} in $V$ if there is a nonnegative $W$ $\in C_0(V)$, $W \not \equiv 0$, such that $h-W \ge 0$ on $C_0(V)$. A positive form $h$ which is not subcritical is called {\bf{critical}} in $V$.
\end{definition}
In \cite[Theorem 5.3]{pinchoverC} it is shown that the criticality of $h$ is equivalent to the existence of a unique positive function which is $H$-harmonic. Such a function is called the ground state of $h$. 
\begin{definition}\label{defopt}
Let $h$ be a quadratic form associated with a Schr\"odinger operator $H$. We say that a positive function  $W: V \to [0, \infty)$ is an {\bf optimal} Hardy weight for $h$ in $V$ if 
\begin{itemize}
    \item $h-W$ is critical in $V$ ({\bf criticality});
    \item $h-W \ge \lambda W$ fails to hold on $C_0(V\setminus K)$ for all $\lambda >0$ and all finite $K \subset V$ ({\bf optimality near infinity});
    \item the ground state $\Psi \notin \ell^2_W$ ({\bf null-criticality}), namely  
    \begin{align*}
        \sum_{x \in V} \Psi^2(x)W(x)= + \infty.
    \end{align*}
\end{itemize}
In the following, for shortness, we will say that $H$ is critical if and only if its associated quadratic form $h$ is critical.
\end{definition}

Finally, we recall that a function $u:V \to \mathbb{R}$ is proper on $V$ if $u^{-1}(K)$ is finite for all compact sets $K\subset u(V)$.
\subsection{Hardy-type inequalities on $\mathbb{T}_{q+1}$} In this subsection we shall state various Hardy-type inequalities on the homogeneous tree $\mathbb{T}_{q+1}$ with $q\geq 2$ \nero. We start with an optimal inequality for $\Delta$ obtained by combining the explicit formula of the Green function and \cite[Theorem 0.2]{pinchover}.
\begin{proposition}\label{optimal1}
For all $\varphi \in C_0(\mathbb{T}_{q+1})$ the following inequality holds:
\begin{align*}
     \frac{1}{2} \sum\limits_{\substack{x,y \in \mathbb{T}_{q+1} \\ 
      x \sim y}}\bigg(\varphi(x)-\varphi(y) \bigg)^2 \ge  \sum_{x \in \mathbb{T}_{q+1}}W_{opt}(x) \varphi^2(x),
\end{align*}
 where 
 \begin{align}\label{numerostella}
  W_{opt}(x)=\begin{cases} \Lambda_q+q^{1/2}-q^{-1/2} &\text{if $|x|=0$,} \\ 
 \Lambda_q &\text{if $|x| \ge 1$.} \end{cases}
  \end{align}
Furthermore, the weight $W_{opt}$  is optimal for $\Delta$.
\end{proposition}

\begin{remark}\label{groundstate}
As a consequence of the results of \cite[Theorem 0.2]{pinchover} it follows that $G^{1/2}$ is the ground state of $h_{\Delta}-W_{opt}$. Furthermore, it is readily checked that
\begin{align*}\label{numero}
    \sum_{x \in \mathbb{T}_{q+1}} G(x)W_{opt}(x)=+\infty\,,
\end{align*}
namely $G^{1/2}\notin \ell^2_{W_{opt}}$.
\end{remark}
In the next theorem we state a family of Hardy-type inequalities depending on two parameters $\beta, \gamma$. The weights $W_{\beta,\gamma}$ provided can be seen as a generalization of $W_{opt}$. Indeed, if we fix $\beta=0$ and $\gamma=1$ in the statement below, we obtain $W_{opt}$.
\begin{theorem}\label{alphabeta} For all $0 \le \beta \le \log_2 q^{1/2}$ and $ q^{-1/2} \le \gamma \le q^{1/2}+q^{-1/2}-2^\beta$ the following inequality holds 
\begin{equation*}
     \frac{1}{2} \sum\limits_{\substack{x,y \in \mathbb{T}_{q+1} \\ x \sim y}}\bigg(\varphi(x)-\varphi(y) \bigg)^2 \ge  \sum_{x \in \mathbb{T}_{q+1}}W_{ \beta, \gamma}(x) \varphi^2(x)\qquad 
    \forall      \varphi \in C_0(\mathbb{T}_{q+1})\,,
\end{equation*}
where the $W_{\beta,\gamma} \ge 0$ are defined as follows:
\begin{align*}
    W_{ \beta, \gamma}(x) = \begin{cases} q+1-q^{1/2}(\frac{1}{\gamma}+\frac{1}{q\gamma}) &\text{if $|x| = 0$, } \\ 
q+1-{q^{1/2}(2^\beta+\gamma}) &\text{if $|x|=1$,} \\
q+1-{q^{1/2}[(1+\frac{1}{|x|})^\beta}+(1-\frac{1}{|x|})^\beta] &\text{if $|x|\ge 2$.} \end{cases}
\end{align*}
\end{theorem}
\begin{remark}
Notice that 
\begin{align*}
   W_{\beta,\gamma}(x)=\Lambda_q+q^{1/2}\frac{\beta (1-\beta)}{|x|^2}+o\Big{(}\frac{1}{|x|^2}\Big{)} \qquad \text{as $|x| \to \infty$.} 
\end{align*}
Since 
\begin{equation*}
    \max_{\beta} \beta(1-\beta)=1/4,
\end{equation*}
which is reached for $\beta=1/2$, $W_{1/2,\gamma}$ is the largest among the $W_{\beta,\gamma}$ at infinity.\par
On the other hand, in order to maximize the value of $W_{\beta,\gamma}$ at $o$, $\gamma$ has to be taken as large as possible, namely $\gamma=q^{-1/2}+q^{1/2}-2^\beta$. Since this quantity is maximum for $\beta=0$, the largest weight at $o$ is $W_{0,\overline{\gamma}}$ with $\overline{\gamma}=q^{-1/2}+q^{1/2}-1$. Notice that: $W_{0,\overline{\gamma}}\equiv W_{opt}$ for $|x| \ge 2$, while $W_{0,\overline{\gamma}}(o)>W_{opt}(o)$ and $W_{opt}(|x|=1)>W_{0,\overline{\gamma}}(|x|=1)$, hence the two weights are not globally comparable.
\nero
\end{remark}
The previous remark suggests that, in order to have the largest weight at infinity, one has to fix $\beta=1/2$ in Theorem \ref{alphabeta}. This intuition is somehow confirmed by the statement below.  \begin{theorem}\label{optimalHf}For all $q^{-1/2} \leq \gamma \leq q^{-1/2}+q^{1/2}-2^{1/2}$ the following inequality holds
\begin{align*}
  \frac 12 \sum\limits_{\substack{x,y \in \mathbb{T}_{q+1} \\ x \sim y}}\bigg{(}\varphi(x)-\varphi(y)\bigg{)}^2\ge \sum_{x \in \mathbb{T}_{q+1}}{W_{1/2,\gamma}}(x)\varphi^2(x) \qquad \forall \varphi \in C_0(\mathbb{T}_{q+1}),
\end{align*}
where 
\begin{align*}
   {W_{1/2,\gamma}}(x) = \begin{cases}  
   q+1-q^{1/2}(\frac{1}{\gamma}+\frac{1}{q\gamma}) &\text{if $|x| = 0$, } \\ 
q+1-{q^{1/2}(2^{1/2}+\gamma}) &\text{if $|x|=1$,} \\ 
    q+1-q^{1/2}[(1+\frac{1}{|x|})^{1/2}+(1-\frac{1}{|x|})^{1/2}] &\text{if $|x| \ge 2$}.
    \end{cases}
\end{align*}
Furthermore, the weights ${W_{1/2,\gamma}}$ are optimal Hardy weights for $\Delta$ in the sense of Definition \ref{defopt}.
\end{theorem}
Using the same argument it is also possible to show that the weights we obtained in Theorem \ref{alphabeta} are optimal near infinity, i.e. the constant is sharp in $\mathbb{T}_{q+1} \setminus K$ for every compact set $K$. 
\begin{corollary}\label{sharp}
For all $0 \le \beta < \min\{\log_2 q^{1/2},1\}$ and $q^{-1/2}\leq \gamma \leq q^{-1/2}+q^{1/2}-2^{\beta}$  the following inequality holds  
\begin{align}\label{sharpineq}
     \frac{1}{2} \sum\limits_{\substack{x,y \in \mathbb{T}_{q+1} \\ x \sim y}}\bigg(\varphi(x)-\varphi(y) \bigg)^2 \ge  \sum_{x \in \mathbb{T}_{q+1}}W_{ \beta,\gamma}(x) \varphi^2(x)\qquad \forall \varphi\in C_0(\mathbb{T}_{q+1}).
\end{align}
Moreover, the constant $1$ in front of the r.h.s. term is sharp at infinity, in the sense that inequality \eqref{sharpineq} fails on $C_0(\mathbb{T}_{q+1}\setminus K)$ if we replace $W_{\beta,\gamma}$ with $CW_{\beta,\gamma}$, for all $C>1$ and all compact set $K$. 
\end{corollary}
\subsection{Improved Poincar\'e inequalities}
We shall provide three examples of improved Poincar\'e inequalities derived by the Hardy-type inequalities stated in the previous subsection. We recall that the Poincar\'e inequality on $\mathbb{T}_{q+1}$ writes
\begin{align}\label{punto}
    \frac12 \sum\limits_{\substack{x,y \in \mathbb{T}_{q+1} \\  x \sim y}} \bigg{(}\varphi(x)-\varphi(y)\bigg{)}^2 \ge \Lambda_q \sum_{x \in \mathbb{T}_{q+1}} \varphi^2(x) \qquad \forall \varphi \in C_0(\mathbb{T}_{q+1}),
\end{align}
and the constant $\Lambda_q$ is sharp in the sense that the above inequality cannot hold with a constant $\Lambda > \Lambda_q$.

The following improved Poincar\'e inequality is an immediate consequence of Theorem \ref{optimal1}.
\begin{proposition}
The following inequality holds 
\begin{align}\label{asterisco}
  \frac 12 \sum\limits_{\substack{x,y \in \mathbb{T}_{q+1} \\ x \sim y}}\bigg{(}\varphi(x)-\varphi(y)\bigg{)}^2\ge \Lambda_q\sum_{x \in \mathbb{T}_{q+1}} \varphi^2(x) + \sum_{x \in \mathbb{T}_{q+1}} R_{q}(x) \varphi^2(x) \qquad \forall \varphi \in C_0(\mathbb{T}_{q+1}),
\end{align}
where
\begin{align*}
    R_{q}(x) = \begin{cases} q^{1/2}-q^{-1/2} &\text{if $|x|=0$,}
\\
0 &\text{otherwise.}
\end{cases} 
\end{align*}
Furthermore, the operator $\Delta-\Lambda_q-R_q$ is critical, hence the inequality does not hold with any $R>R_q$.
\end{proposition}
Notice that \eqref{asterisco} improves \eqref{punto} only locally, namely at $o$. The next statement provides a global improvement of \eqref{punto}.
\begin{theorem}\label{POMIGLIORATA} For all $0 \le \beta \le \log_2\big{(}\frac{3}{2}-\frac{1}{2q}\big{)}$ and $\frac12 + \frac{1}{2q} \le \gamma \le 2-2^\beta $, it holds
\begin{align}\label{asterisco2}
   \frac 12 \sum\limits_{\substack{x,y \in \mathbb{T}_{q+1} \\ x \sim y}}\bigg{(}\varphi(x)-\varphi(y)\bigg{)}^2\ge \Lambda_q\sum_{x \in \mathbb{T}_{q+1}} \varphi^2(x) + \sum_{x \in \mathbb{T}_{q+1}} R_{\beta, \gamma} \varphi^2(x) \qquad \forall \varphi \in C_0(\mathbb{T}_{q+1}),
\end{align}
where  
\begin{align*}
  0 \le R_{\beta,\gamma}(x) = \begin{cases} q^{1/2}(2-\frac{1}{\gamma}-\frac{1}{q\gamma}) &\text{if $|x|=0$,}\\
q^{1/2}(2-2^{\beta}-\gamma) &\text{if $|x|=1$,} \\ 
q^{1/2}\bigg{(}2-(1+\frac{1}{|x|})^\beta-(1-\frac{1}{|x|})^\beta\bigg{)} &\text{if $|x| \ge 2$}.
\end{cases}
\end{align*}
\end{theorem}
Notice that \eqref{asterisco2} improves globally \eqref{punto} but it gives no information about the sharpness of $R_{\beta,\gamma}$. A sharp improvement is instead provided by the next theorem which holds for functions supported outside the ball $B_2(o)$.
\begin{theorem}\label{2.11}  The following inequality holds
\begin{align}\label{optimalPoincar}
   \frac 12 \sum\limits_{\substack{x,y \in \mathbb{T}_{q+1} \\ x \sim y}}\bigg{(}\varphi(x)-\varphi(y)\bigg{)}^2\ge \sum_{x \in \mathbb{T}_{q+1}}\Lambda_q \varphi^2(x)+\sum_{x \in \mathbb{T}_{q+1}}\overline{R}(x)\varphi^2(x) \ \ \ \ \ \ \forall \varphi \in C_0(\mathbb{T}_{q+1}\setminus B_2(o)),
\end{align}
 where 
\begin{align*}
   \overline{R}(x) =
    q^{1/2}\bigg{[}2-\bigg{(}1+\frac{1}{|x|}\bigg{)}^{1/2}-\bigg{(}1-\frac{1}{|x|}\bigg{)}^{1/2}\bigg{]} \qquad \text{if $|x| \ge 2$.}
\end{align*}
Moreover, the constant $q^{1/2}$ is sharp in the sense that inequality \eqref{optimalPoincar} cannot hold if we replace the remainder term $\overline{R}$ with $C\bigg{[}2-(1+\frac{1}{|x|})^{1/2}-(1-\frac{1}{|x|})^{1/2}\bigg{]}$ and $C>q^{1/2}$.
\end{theorem}
\section{Proofs of the results}
We collect here the proofs of the results stated in Section 2. 
\subsection{Proofs of Hardy-type  inequalities}
\begin{proof}[Proof of Proposition \ref{optimal1}]  Consider the function $\Tilde{u}(x)= \sqrt{G(x,o)}$, where $G$ is the Green function on $\mathbb{T}_{q+1}$. By \cite[Theorem  0.2]{pinchover} we only need to show that
\begin{equation*}
   \frac{\Delta {\Tilde{u}(x)}}{{\Tilde{u}(x)}}=W_{opt}(x).
\end{equation*}
By the explicit formula for the Green function on $\mathbb{T}_{q+1}$ given in \cite[Lemma 1.24]{Woess} we have
\begin{equation*}
    \Tilde{u}(x) = \sqrt{\frac{q}{q-1}\bigg(\frac{1}{q}\bigg{)}^{|x|}}.
\end{equation*}
For $x \ne o$, we obtain that 
\begin{align*}
  \frac{\Delta \Tilde{u}(x)}{\Tilde{u}(x)} &= \bigg( q+1-\bigg(\frac{1}{q} \bigg)^{1/2} - q\bigg(\frac{1}{q} \bigg)^{-1/2}\bigg)= \big({q^{1/2}} -1\big)^2 = \Lambda_q.
  \end{align*}
For $x=o$ we get
  \begin{align*}
       \frac{\Delta \Tilde{u}(o)}{\Tilde{u}(o)} &= \Bigg[(q+1)  \bigg{(}{\frac{ {q}}{q-1}}\bigg{)}^{1/2} - (q+1) \Bigg{(}\frac{1}{({q-1})^{1/2}} \Bigg{)} \Bigg]  \bigg{(}{\frac{q-1}{q}}\bigg{)}^{1/2} \\ &= q+1 -\frac{q+1}{{q}^{1/2}}=\Lambda_q+q^{1/2}-q^{-1/2} > \Lambda_q.
  \end{align*}
\end{proof}
\bigskip
\begin{proof}[Proof of Theorem \ref{alphabeta}] The statement follows from \cite[Proposition 3.1]{Golenia} by providing a suitable positive super-solution to the equation $\Delta u=W_{\beta,\gamma} u$ in $\tq$. To this aim, we define the function:
\begin{align}\label{numerocerch}
    u_{\beta, \gamma}(x) = \begin{cases} q^{{-|x|/2}}|x|^\beta &\text{if $|x| \ge 1$, } \\ 
\gamma &\text{if $|x|=0$.} \end{cases}
\end{align}
Now, by writing $u=u_{ \beta, \gamma}$, we have
\begin{align*}
    \frac{\Delta u(o)}{u(o)} &= q+1-(q+1)\frac{q^{-1/2}}{\gamma}=q+1-q^{1/2}\Big(\frac{1}{\
    \gamma}+\frac{1}{q\gamma }\Big), 
    \end{align*}
    which is nonnegative if $\gamma \ge {q}^{-1/2}$.\\
    {Next, for every $x$ such that $|x|=1$, we have} 
    \begin{align*}
    \frac{\Delta u(x)}{u(x)}&= q+1-q\frac{q^{-1}2^\beta}{q^{-1/2}} - \frac{\gamma}{q^{-1/2}} =q+1-{q^{1/2}(2^\beta+\gamma}), 
    \end{align*}
    which is nonnegative if  $\gamma \le q^{1/2}+{q}^{-1/2}-2^\beta $. The restriction $\beta \le 1/2 \log_2 q$ comes out to make consistent $q^{-1/2} \le \gamma \le q^{-1/2}+q^{1/2} -2^\beta$.\\ 
    Finally, for every $x$ such that $|x|\ge 2$, we have
    \begin{align}\label{x>=2}
    \frac{\Delta u(x)}{u(x)}&=q+1-q\frac{q^{-(|x|+1)/2}(|x|+1)^\beta}{q^{-|x|/2}|x|^\beta} - \frac{q^{-(|x|-1)/2}(|x|-1)^\beta}{q^{-|x|/2}|x|^\beta} \nonumber \\ &=q+1-{q^{1/2}\bigg[\bigg{(}1+\frac{1}{|x|}\bigg{)}^\beta}+\bigg{(}1-\frac{1}{|x|}\bigg{)}^\beta\bigg] \ge 0. 
\end{align}
 If $\beta\le1$, then the function $f:\mathbb{R}^+ \to \mathbb{R}$ defined by $f(x)=x^\beta$ is concave. It follows that 
\begin{align*}
    f\bigg{(}\frac{1}{2}\bigg{(}1+\frac{1}{|x|}\bigg{)}+\frac12\bigg{(}1-\frac{1}{|x|}\bigg{)}\bigg{)}=f(1) \ge \frac 12 f\bigg{(}1+\frac{1}{|x|}\bigg{)}+\frac 12f\bigg{(}1-\frac{1}{|x|}\bigg{)},
\end{align*}
that is equivalent to 
\begin{equation*}
    2 \ge \bigg{(}1+\frac{1}{|x|}\bigg{)}^\beta + \bigg{(}1-\frac{1}{|x|}\bigg{)}^\beta.
\end{equation*}
Then, 
\begin{equation*}
    \frac{\Delta u(x)}{u(x)} =q+1-q^{1/2}\bigg{[}\bigg{(}1+\frac{1}{|x|}\bigg{)}^\beta + \bigg{(}1-\frac{1}{|x|}\bigg{)}^\beta\bigg{]} \ge q+1-2q^{1/2}=\Lambda_q >0\qquad \forall |x|\geq 2,
\end{equation*}
which proves \eqref{x>=2}. 

If $\log_2 q^{1/2} \ge \beta >1$, notice that the function $h : [2, + \infty) \to \mathbb{R}$ defined by $h(x)=(1+\frac{1}{x})^\beta+(1-\frac{1}{x})^\beta$ is decreasing. 
 Then $h$ reaches its maximum at 2. Thus to show \eqref{x>=2} it suffices to prove that
\begin{align}\label{secondocaso}
h(x)\leq h(2)=    \bigg{(}\frac{3}{2}\bigg{)}^\beta+\bigg{(}\frac{1}{2}\bigg{)}^\beta \le q^{1/2}+q^{-1/2}.
\end{align}
Notice that for every $\beta\geq 1$ 
\begin{align*}
    \frac{{\rm{d}}}{ {\rm{d}} \beta} \bigg{[} \bigg{(}\frac{3}{2}\bigg{)}^\beta+\bigg{(}\frac{1}{2}\bigg{)}^\beta \bigg{]} =2^{-\beta}(3^\beta \log(3/2)-\log(2)) \ge 0.
\end{align*}
Hence  
\begin{align*}
\bigg{(}\frac{3}{2}\bigg{)}^\beta+\bigg{(}\frac{1}{2}\bigg{)}^\beta \le    \bigg{(}\frac{3}{2}\bigg{)}^{\log_2 q^{1/2}}+\bigg{(}\frac{1}{2}\bigg{)}^{\log_2 q^{1/2}} \le 2^{\log_2 q^{1/2}}+2^{-\log_2 q^{1/2}}=q^{1/2}+q^{-1/2},
\end{align*}
so that \eqref{secondocaso} holds and the proof is concluded.
\end{proof}
\begin{remark}
Note that the statement of Theorem $\ref{alphabeta}$ can be enriched by considering the family of radial functions
\begin{align*}
    u_{\alpha,\beta, \gamma}(x) = \begin{cases} q^{{\alpha}|x|}|x|^\beta &\text{if $|x| \ge 1$, } \\ 
\gamma &\text{if $|x|=0$,} \end{cases}
\end{align*}
with $\alpha \in \mathbb{R}$ and $\beta$ and $\gamma$ as in Theorem $\ref{alphabeta}$. Indeed, a straightforward computation shows that for $|x| \ge 2$ 
\begin{align*}
    W_{\alpha,\beta, \gamma}(x)=\frac{\Delta u_{\alpha,\beta,\gamma}(x)}{u_{\alpha,\beta,\gamma}(x)}=q+1-{q^{\alpha+1}\bigg(1+\frac{1}{|x|}\bigg)^\beta}-q^{-\alpha}\bigg(1-\frac{1}{|x|}\bigg)^\beta.
\end{align*}
Nevertheless,  
\begin{align*}
    W_{\alpha,\beta, \gamma}(x)= q+1-q^{1+\alpha}-q^{-\alpha}+ o(1) \qquad \text{as $|x| \to +\infty$},
\end{align*} 
which is maximum for $\alpha=-1/2$. Therefore, the choice $\alpha=-1/2$ turns out to be the best to get a weight as large as possible at $\infty$.
\end{remark}

We shall now prove our main result, i.e. Theorem \ref{optimalHf}.
\begin{proof}[Proof of Theorem \ref{optimalHf}]
Consider the Schr\"odinger operator $H:=\Delta+Q$, with 
\begin{align*}
    Q(x)=\begin{cases} 0 &\text{if $|x|=0,$} \\ 
    q^{1/2} &\text{if $|x|=1$,} \\
    -\Lambda_q &\text{if $|x| \ge 2$.}
    \end{cases}
\end{align*}
{{\it Step 1.}} We construct an optimal Hardy weight for $H$. To this aim, we exploit \cite[Theorem 1.1]{pinchover} that provides an optimal Hardy weight for  a Schr\"odinger operator $H$ by using $H$-harmonic functions.\\ For the sake of completeness we start by briefly recalling the statement of \cite[Theorem 1.1]{pinchover}: \\ given two positive $H$-superharmonic functions $u,v$ which are $H$-harmonic outside a finite set, if the function $u_0:=u/v$ is proper and $\sup_{x\sim y} u_0(x)/u_0(y) < + \infty$, then $\widetilde{W}:=\frac{H[(uv)^{1/2}]}{(uv)^{1/2}}$ is an optimal weight for $H$.\\ 
Next we define 
\begin{align*}
    u(x)&:=\begin{cases} {\gamma} &\text{if $|x|=0$,} \\
    q^{-|x|/2} &\text{if $|x| \ge 1$,}
    \end{cases} \\ 
    v(x)&:= \begin{cases} {\gamma} &\text{if $|x|=0$,} \\ 
    |x|q^{-|x|/2} &\text{if $|x| \ge 1$.}
    \end{cases}
\end{align*} Now we show that $u,v$ satisfy the hypothesis of the above-mentioned theorem.
\\  Indeed,
\begin{align*}
    &Hu(o)=(q+1)({\gamma}-q^{-1/2})+Q(o){\gamma} \ge 0, \\ 
    &Hv(o)=(q+1)({\gamma}-q^{-1/2})+Q(o){\gamma} \ge 0.
    \end{align*} 
    If $|x|=1$, then
    \begin{align*}
     Hu(x)&=(q+1)q^{-1/2}-qq^{-1}-{\gamma} + q^{-1/2}q^{1/2}= q^{1/2}+q^{-1/2}-{\gamma} \ge 2^{1/2}, \\ 
    Hv(x)&=(q+1)q^{-1/2}-2q^{-1}q-{\gamma} + Q(x)q^{-1/2} \\ &\ge q^{1/2}+q^{-1/2}-2-q^{-1/2}-q^{1/2}+ 2^{1/2}+1 =2^{1/2}+1-2 > 0.
\end{align*}
If $|x| \ge 2$, then
\begin{align*}
   Hu(x)&= (q+1)q^{-|x|/2}-qq^{-(|x|+1)/2}-q^{-(|x|-1)/2}-\Lambda_qq^{-|x|/2} \\ 
   &=q^{-|x|/2}(q+1-2q^{1/2}-\Lambda_q)=0; \\ 
    Hv(x) &= (q+1)|x|q^{-|x|/2}-(|x|+1)qq^{-(|x|+1)/2}-(|x|-1)q^{-(|x|-1)/2} -\Lambda_q q^{-|x|/2} \\ 
    &=|x|q^{-|x|/2}(q+1-2q^{1/2}-\Lambda_q)=0.
\end{align*}
Define now
\begin{align*}
    u_0(x):=\frac{u(x)}{v(x)}=\begin{cases}  1 &\text{if $|x|=0$,} \\ 
    \frac{1}{|x|} &\text{otherwise.} \end{cases}
\end{align*}
The function $u_0$ is proper  because 
    $\lim_{|x| \to \infty} u_0(|x|) = 0$
and $u_0(|x|)>u_0(|x|+1)>0$ for all $|x| \ge 1$, thus $u_0^{-1}(K)$ is finite for all compact set $K \subset (0, \infty)$. \\ 
Now consider $x \sim y$ and compute
\begin{align*}
    \frac{u_0(x)}{u_0(y)}= \begin{cases}   1 &\text{if $|x|=0$,} \\ {1}/{{\gamma}} &\text{if $|y|=0$ and $|x|=1$,} \\ 
    1+\frac{1}{|x|} &\text{if $|y|=|x|+1$ and $|x|\ge 1$,} \\
    1-\frac{1}{|x|} &\text{if $|y|=|x|-1$ and $|x| \ge 2$.}
    \end{cases}
\end{align*}
Thus 
   $ \sup\limits_{\substack{x,y \in \mathbb{T}_{q+1} \\ x \sim y}} \frac{u_0(x)}{u_0(y)} < + \infty.$
Hence, from \cite[Theorem 1.1]{pinchover} we conclude that the weight
\begin{align*}
    \widetilde{W}(x):&=\frac{H[(uv)^{1/2}](x)}{(uv)^{1/2}(x)}= \frac{\Delta(uv)^{1/2}(x)}{(uv)^{1/2}(x)}+ Q(x) \\ 
    &=\begin{cases} (q+1)(1-\frac{q^{-1/2}}{{\gamma}}) &\text{if $|x|=0$,} \\ 
  (q+1)-q^{1/2}(2^{1/2}+{\gamma})+q^{1/2} &\text{if $|x|=1$,} \\ 
   (q+1)-q^{1/2}[(1+\frac{1}{|x|})^{1/2}+(1-\frac{1}{|x|})^{1/2}]-\Lambda_q &\text{if $|x| \ge 2$}
   \end{cases}
\end{align*}
is an optimal weight for $H$. \\ 

\smallskip

{\it Step 2.} We derive an optimal Hardy weight for $\Delta$. To this aim we prove that the three conditions of Definition \ref{defopt} are satisfied by the operator $\Delta-W_{1/2,\gamma}$, where $W_{1/2,\gamma}:=\widetilde{W}-Q$.

\smallskip

$\bullet$ {\it Criticality:} the optimal Hardy inequality, obtained considering the quadratic form $h$ associated with $H$, namely
\begin{align*}
    \frac 12 \sum\limits_{\substack{x,y \in \mathbb{T}_{q+1} \\ x \sim y}}\bigg{(}\varphi(x)-\varphi(y)\bigg{)}^2+\sum_{x \in \mathbb{T}_{q+1}}Q(x)\varphi^2(x)\ge \sum_{x \in \mathbb{T}_{q+1}}\bigg{(}\frac{\Delta(uv)^{1/2}(x)}{(uv)^{1/2}(x)}+ Q(x)\bigg{)}\varphi^2(x)
\end{align*}
is equivalent to the Hardy inequality associated to $\Delta$
\begin{align*}
      \frac 12 \sum\limits_{\substack{x,y \in \mathbb{T}_{q+1} \\ x \sim y}}\bigg{(}\varphi(x)-\varphi(y)\bigg{)}^2 \ge \sum_{x \in \mathbb{T}_{q+1}}\frac{\Delta(uv)^{1/2}(x)}{(uv)^{1/2}(x)}\varphi^2(x) \qquad \forall \varphi \in C_0(\mathbb{T}_{q+1}).
\end{align*}
Moreover,
\begin{align*}
    {W_{1/2,\gamma}}(x):=\frac{\Delta(uv)^{1/2}(x)}{(uv)^{1/2}(x)}= \begin{cases} q+1-q^{1/2}(\frac{1}{{\gamma}}+\frac{1}{q\gamma}) &\text{if $|x|=0$,} \\ 
     q+1-q^{1/2}(2^{1/2}+{\gamma}) &\text{if $|x|=1$,} \\
       q+1-q^{1/2}[(1+\frac{1}{|x|})^{1/2}+(1-\frac{1}{|x|})^{1/2}] &\text{if $|x| \ge 2$,}
      \end{cases}
\end{align*}
is nonnegative. The optimality of $\widetilde{W}$ for $H$ implies that it does not exist a nonnegative function $f \not \equiv 0$  such that 
\begin{align*}
    \frac 12 \sum\limits_{\substack{x,y \in \mathbb{T}_{q+1} \\ x \sim y}}\bigg{(}\varphi(x)-\varphi(y)\bigg{)}^2 - \sum_{x \in \mathbb{T}_{q+1}}{W_{1/2,\gamma}}(x)\varphi^2(x) \ge \sum_{x \in \mathbb{T}_{q+1}}f(x) \varphi^2(x),
\end{align*}
or, equivalently, $\Delta - {W_{1/2,\gamma}}$ is critical.  

\smallskip

$\bullet$ {\it Null-criticality} : the function $z=(uv)^{1/2}$ is the ground state of $h_{\Delta} - {W_{1/2,\gamma}}$. Notice that 
\begin{align*}
  {W_{1/2,\gamma}}(x)> W_{opt}(x) \qquad \text{if $|x| \ge 2$},  \\ 
   z(x)>G^{1/2}(x) \qquad \text{if $|x| \ge 2$}, 
\end{align*}
where $W_{opt}$ is defined by \eqref{numerostella} and $G$ is the Green function. Then by Remark \ref{groundstate}
\begin{align*}
    \sum_{x \in \mathbb{T}_{q+1}} z^2(x){W_{1/2,\gamma}}(x) = +\infty.
\end{align*} 

\smallskip

$\bullet$ {\it Optimality near infinity} : suppose by contradiction that there exist $\overline{\lambda}>0$ and a compact set $K\subset\mathbb{T}_{q+1}$ such that 
\begin{align} \label{OPTINF}
     \frac 12 \sum\limits_{\substack{x,y \in \mathbb{T}_{q+1} \\ x \sim y}}\bigg{(}\varphi(x)&-\varphi(y)\bigg{)}^2-\sum_{x \in \mathbb{T}_{q+1} }{W_{1/2,\gamma}}(x)\varphi^2(x)  \ge \overline{\lambda}\sum_{x \in \mathbb{T}_{q+1}} {W_{1/2,\gamma}}(x)\varphi^2(x),
\end{align}
for all $\varphi \in C_0(\mathbb{T}_{q+1} \setminus K)$. Then, \eqref{OPTINF} holds true on $C_0(\mathbb{T}_{q+1}\setminus(K\cup B_2(o)))$. Notice that $W_{opt}\varphi^2 \le {W_{1/2,\gamma}}\varphi^2$ for all $\varphi \in C_0(\mathbb{T}_{q+1}\setminus(K\cup B_2(o)))$.
It follows that
\begin{align*}
      \frac 12 \sum\limits_{\substack{x,y \in \mathbb{T}_{q+1} \\ x \sim y}}\bigg{(}\varphi(x)-\varphi(y)\bigg{)}^2-\sum_{x \in \mathbb{T}_{q+1}}{W_{opt}}(x)\varphi^2(x) &\ge \frac12 \sum\limits_{\substack{x,y \in \mathbb{T}_{q+1} \\ x \sim y}}\bigg{(}\varphi(x)-\varphi(y)\bigg{)}^2-\sum_{x \in \mathbb{T}_{q+1}}{W_{1/2,\gamma}}(x)\varphi^2(x) \\ &\ge \overline{\lambda}\sum_{x \in \mathbb{T}_{q+1}}{W_{1/2,\gamma}}(x)\varphi^2(x) \ge \overline{\lambda}\sum_{x \in \mathbb{T}_{q+1}}
      {W_{opt}}(x)\varphi^2(x),
\end{align*}
for all $\varphi \in C_0(\mathbb{T}_{q+1}\setminus(K\cup B_2(o)))$. This is a contradiction because $W_{opt}$ is optimal for $\Delta$. We checked the three conditions given in Definition \ref{defopt}. Hence $W_{1/2,\gamma}$ is optimal for $\Delta$.
\end{proof}
\begin{proof}[Proof of Corollary \ref{sharp}] For $\beta < \min\{ 1/2 \log q, 1 \}$ we have that $W_{\beta,\gamma}>W_{opt}$ on $B_2(o)^c$. Then, the thesis follows by repeating the same argument used for proving \eqref{OPTINF}.
\end{proof}
\subsection{Proof of improved Poincar\'e inequalities}
\begin{proof}[Proof of Theorem \ref{POMIGLIORATA}]
Given $W_{\beta, \gamma}=\frac{\Delta u_{\beta, \gamma}}{u_{\beta,\gamma}}$, where $u_{\beta,\gamma}$ is defined by \eqref{numerocerch}, it is easy to check that $W_{\beta,\gamma}$ is larger than $\Lambda_q$ on $B_2(o)$ choosing the parameters  $0 \le \beta \le \log_2\big{(}\frac{3}{2}-\frac{1}{2q}\big{)}$ and $\frac12 + \frac{1}{2q} \le \gamma \le 2-2^\beta $ . \\
Indeed,
\begin{align*}
    q+1-(q+1)q^{-1/2}/\gamma \ge q+1-2q^{1/2} 
    \end{align*}
    is equivalent to  $\frac12+\frac{1}{2q} \le \gamma $, and 
\begin{align*}
    q+1-q^{1/2}(2^\beta+\gamma) \ge q+1-2q^{1/2} 
    \end{align*} 
 is equivalent to $\gamma \le  2-2^\beta.$
Notice that for this choice of $\gamma$ and $\beta$ it follows that $\beta \le \log_2(\frac32)<1$, and we already proved in Theorem \ref{alphabeta} that $W_{\beta,\gamma}\ge \Lambda_q$ on ${B_2(o)}^c$ for all $0\le\beta <1$.
\end{proof}
\begin{proof}[Proof of Theorem \ref{2.11}]
We know from Theorem \ref{optimalHf} that the optimal weight ${W_{1/2,\gamma}}$ is larger than $\Lambda_q$ for $|x| \ge 2$. Then we can define 
\begin{align*}
    \overline{R}(x)= {W_{1/2,\gamma}}(x)-\Lambda_q \qquad \forall x \in \mathbb{T}_{q+1} \setminus B_2(o),\end{align*}
    and \eqref{optimalPoincar} follows. The sharpness of $q^{1/2}$ is  consequence of the optimality of $\widetilde{W}$ for $H$ where $\widetilde{W}$ and $H$ are chosen such as in the proof of Theorem \ref{optimalHf}. 
\end{proof}
\section{Hardy-type inequalities on rapidly growing radial trees}
In view of the results obtained on the homogeneous tree, here we attempt to generalise the family of Hardy inequalities given in Theorem \ref{optimalHf} on a more general context, namely on radial trees.  

Let $T=(V,E)$ be an infinite tree. We call $T$ a {\bf{radial}} tree if the degree $m$ depends only on $|x|$ (see e.g. \cite{Golenia, Woich}). In the following we set $\overline{m}=m-1$ to lighten the notation. For future purposes, we also note that the volume of the ball $B_n(o)$ is given by
\begin{align*}
    \#B_1(o)&=1, \\ 
    \#B_2(o)&=2+\overline{m}(0), \\
    \#B_3(o)&=2+\overline{m}(0)+(\overline{m}(0)+1)\overline{m}(1) ,\\ 
    \vdots \\
       \#B_n(o)&=1+(\overline{m}(0)+1)[1+\overline{m}(1)+\overline{m}(1)\overline{m}(2)+... +\overline{m}(1)\overline{m}(2)\overline{m}(3)\dots\overline{m}(n-2)].
\end{align*}
If particular, if $T=\tq$, then $\overline{m}\equiv q$ and we have that $\#B_n(o) \sim q^{n-1}$ as $n \rightarrow + \infty$.\par
Next, recalling that the proof of Theorem \ref{alphabeta}  relies on the exploitation of the superharmonic functions $u_{\alpha,\beta}$ and that $u_{\alpha,\beta}(x) =  |x|^\beta q^{\alpha |x|}$ for all $|x| \ge 1$, by analogy, we consider on $T$ the family of positive and radial functions: 
\begin{align}\label{supersol2}
    u_{\alpha, \beta}(x):=|x|^\beta \Psi^\alpha(|x|)\qquad {\rm{if}}\,\,|x|\geq 1\,.
\end{align}
Regarding the choice of the function $\Psi$, since in $\tq$ the function $q^{|x|}$ is related to $\#B_{|x|+1}(o)$ and since $\frac{\#B_{|x|+1}(o)}{\#B_{|x|}(o)} \sim q=\overline{m}$ as $|x| \rightarrow + \infty$, we assume that it satisfies the following condition
\begin{align}\label{condition}
    \Psi(|x|+1)=\overline{m}(|x|)\Psi(|x|) \quad \text{for all } |x| \ge 1.
\end{align}
Clearly, if $T=\tq$, then \eqref{condition} holds by taking $\Psi(|x|)=q^{|x|}$. We note that, conversely, for a given positive $\Psi$, condition \eqref{condition} characterizes the tree we are dealing with through its degree, see Remark \ref{generaltreee} below.

By showing that the function $u_{-1/2, \beta}$ is superharmonic on $T$, we obtain the following result. 

\begin{proposition}\label{radial tree}
Let $\Psi: (0,+\infty) \rightarrow \R$ be a positive function such that the map $(0,+\infty)\ni s \mapsto \frac{\Psi(s+1)}{\Psi(s)}$ is nondecreasing and let $T$ be a radial tree with degree $\overline{m}+1$ satisfying condition \eqref{condition}. Then, for all $\beta<1$ and  $\frac{1}{\Psi^{1/2}(1)} \le\gamma \le \frac{1}{\Psi^{1/2}(1)}\bigg{(}\overline{m}(1)+1-\overline{m}^{1/2}(1)2^{\beta}\bigg{)}$ the following inequality holds
\begin{align*}
     \frac 12 \sum\limits_{\substack{x,y \in {T} \\ x \sim y}}\bigg{(}\varphi(x)-\varphi(y)\bigg{)}^2\ge \sum_{x \in T} W_{\beta,\gamma} \varphi^2(x) \qquad \forall \varphi \in C_0(T)\,,
\end{align*}
where $W_{\beta,\gamma}$ is the positive weight 
\begin{align*}
    W_{\beta,\gamma}(x):=\begin{cases} \overline{m}(0)+1-\,\frac{\overline{m}(0)+1}{\gamma\, \Psi^{1/2}(1)} &\text{if $|x|=0$,}\\ 
    \overline{m}(1)+1-\overline{m}^{1/2}(1)2^\beta- \Psi^{1/2}(1) \gamma &\text{if $|x|=1$,}\\ 
    \overline{m}(|x|)+1-\overline{m}^{1/2}(|x|)\Big{(}1+\frac{1}{|x|}\Big{)}^\beta-\overline{m}^{1/2}(|x|-1)\Big{(}1-\frac{1}{|x|}\Big{)}^\beta &\text{if $|x|\ge 2$.}
    \end{cases}
\end{align*}
\end{proposition}

\begin{remark}\label{generaltreee}
It is readily seen that, by taking $\Psi(s)=q^{s}$ in Proposition \ref{radial tree}, we get $T=\tq$ and we re-obtain Theorem \ref{optimalHf}; however, Proposition \ref{radial tree} gives no information about the criticality of the operator $\Delta-W_{\beta,\gamma}$ on $T$. We also note that condition \eqref{condition} yields rapidly growing trees, such as those generated, for instance, by the maps $\Psi_a(s)=e^{s^a}$ with $a>1$.
\end{remark}


\begin{proof}
The proof follows the same lines of the proof of Theorem \ref{alphabeta}, namely we show that the function $u_{\alpha, \beta}$ in \eqref{supersol2}, with $\alpha=-1/2$ and  $\beta<1$, is superharmonic in $T \setminus B_2(0)$ and that it can be properly extended to $o$ in order to get a superharmonic function on the whole $T$. Hence the statement follows by invoking \cite[Proposition 3.1]{Golenia}. 

If $\beta<1$ and $|x| \ge 2$ we have
\begin{align*}
 \Delta u_{-1/2,\beta}(x)&=\Big(\overline{m}(|x|)+1\Big{)}|x|^\beta \Psi^{-1/2}(|x|)-\overline{m}^{1/2}(|x|)(|x|+1)^\beta \Psi^{-1/2}(|x|)+ \\ &-(|x|-1)^\beta \overline{m}^{1/2}(|x|-1)\Psi^{-1/2}(|x|) \\ 
    &=u_{-1/2,\beta}(x)\bigg{(}\overline{m}(|x|)+1-\overline{m}^{1/2}(|x|)\Big{(}1+\frac{1}{|x|}\Big{)}^\beta-\overline{m}^{1/2}(|x|-1)\Big{(}1-\frac{1}{|x|}\Big{)}^\beta\bigg{)}.\end{align*}
Since by hypothesis the function $\overline{m}$ is nondecreasing, we get
    \begin{align*}
         \Delta u_{-1/2, \beta}(x)&=u_{-1/2,\beta}(x)\bigg{(}\big{(}\overline{m}^{1/2}(|x|))-1\big{)}^2+\overline{m}^{1/2}(|x|)\Big{(}2-\Big{(}1+\frac{1}{|x|}\Big{)}^\beta-\Big{(}1-\frac{1}{|x|}\Big{)}^\beta\Big{)} \\ &+\Big{(}\overline{m}^{1/2}(|x|)-\overline{m}^{1/2}(|x|-1)\Big{)}\Big{(}1-\frac{1}{|x|}\Big{)}^\beta\bigg{)}> 0,
    \end{align*} 
for all $|x|\geq 2$.

 Then we choose $\gamma:=u_{-1/2,\beta}(o)$ such that $\Delta u_{-1/2,\beta}$ is nonnegative in $B_2(o)$. By a direct computation we have
    \begin{align*}
        \Delta u_{-1/2,\beta}(o) = (\overline{m}(0)+1)(\gamma -\Psi^{-1/2}(1)) \ge 0,
    \end{align*}
for $\gamma \ge \Psi^{-1/2}(1)$. Furthermore, for $|x|=1$ we get 
    \begin{align*}
        \Delta u_{-1/2,\beta}(x)&=(\overline{m}(1)+1)\Psi^{-1/2}(1)  -\overline{m}(1)2^{\beta}\Psi^{-1/2}(2)-\gamma \ge 0, 
    \end{align*}
  for $\gamma \le \Psi^{-1/2}(1)\bigg{(}\overline{m}(1)+1-\overline{m}^{1/2}(1)2^{\beta}\bigg{)}$. This concludes the proof. 
\end{proof}

\par\bigskip\noindent
\textbf{Acknowledgments.} The first author is partially supported by the INdAM-GNAMPA 2019 grant ``Analisi spettrale per operatori ellittici con condizioni di Steklov o parzialmente incernierate'' and by the PRIN project ``Direct and inverse problems for partial differential equations: theoretical aspects and applications'' (Italy). The first and third authors are members of the Gruppo Nazionale per l'Analisi Matematica, la Probabilit\`a e le loro Applicazioni (GNAMPA) of the Istituto Nazionale di Alta Matematica (INdAM).

\end{document}